\def \version {\today}

\documentclass[a4paper,12pt]{article}
\usepackage[utf8]{inputenc}
\usepackage{amsmath, amssymb, amsthm}

\usepackage{multirow}
\usepackage[blocks]{authblk}

\usepackage{lineno}
\newtheorem{theorem}{Theorem}[section]
\newtheorem{definition}[theorem]{Definition}
\newtheorem{corollary}[theorem]{Corollary}
\newtheorem{lemma}[theorem]{Lemma}
\newtheorem{proposition}[theorem]{Proposition}

\newcommand{\Leg}{\operatorname{Leg}}%
\def \bsk {\bigskip}

\def \epsk {\epsilon_{n,k}}

\makeatletter
\def\blfootnote{\gdef\@thefnmark{}\@footnotetext}
\makeatletter

\newcommand{\msc}[2]{\blfootnote{\textup{#1} \textit{Mathematics Subject Classification.} #2.}}
\title{Induced/Incomparable versus Ramsey}

\author[1]{Yair Caro}
\affil[1]{Department of Mathematics, University of Haifa-Oranim, Israel}
\author[2,3]{Zsolt Tuza}
\affil[2]{HUN-REN Alfr\'ed R\'enyi Institute of Mathematics, Budapest, Hungary}
\affil[3]{Department of Computer Science and Systems Engineering, University of Pannonia, Veszpr\'em, Hungary}
\author[4]{Christina Zarb}
\affil[4]{Department of Mathematics, University of Malta, Malta}
\date{\small Latest update on \version }

\begin{document}

\maketitle
\msc{2020}{05C55} 
\begin{abstract}
We consider the following problem: Let $H$ and $F$ be two graphs on $k$ vertices and assume $F \neq H$. We say that $H$ and $F$ are incomparable if neither $F$ nor $H$ contains the other.

Let $H$ be a graph on $k$ vertices and let $G$ be a graph on at least $k$ vertices.  Then $G$ is said to be $H$-exact if any induced subgraph of $G$ on $k$ vertices is  either isomorphic to $H$ or incomparable with $H$.  Exact($H$) is the family of all graphs $G$ which are $H$-exact.

We pose the following problem:  For a graph $H$ on $k$ vertices, determine or estimate  $f(H) = \max \{n: \exists G \in \text{Exact}(H), |V (G)| = n\}$.

Among the many results obtained in this paper the following are representatives concerning trees and matchings.
\begin{enumerate}
\item{For a tree on $k \geq 3$  vertices, $ (k - 1)(\left \lceil \frac{k}{2} \right  \rceil -1 ) \leq   f(T) \leq ( k-1)^2$.}
\item{For $k \geq 4$,  $f(K_{1,k-1}) = (k-1)(k-2)$.}
\item{For $k \geq 5$, $f(P_k) = \frac{(k-1)^2}{2}$ if $k$ is odd and $f(P_k)  = \frac{(k-1)(k-2)}{2}+1$ if $k$ is even.}
\item{$f(nK_2) = 3n$ for $n = 2, 3$ and  $f(nK_2) = 4n - 4$ for $n \geq 4$.}
\end{enumerate}
\end{abstract}
\section{Introduction}
An alternative formulation for the Ramsey numbers $R(H,K_k)$ is that every graph $G$ on at least $R(H,K_k)$ vertices either contains $H$ or has independence number $\alpha(G) \geq k$, that is it contains an induced $E_k=\overline{K_k}$.

However, it is no longer true that $R(H,F)$, where $F \notin \{K_k,E_k\}$, forces that every graph $G$ on at least $R(H,F)$ vertices contains either $H$ or $\overline{F}$, as evident by red $G=E_n$ versus blue $\overline{G}=K_n$ for every $n$.  

We put forward two problems concerning this situation, which we shall present after introducing some necessary terminology.

Let $A$ and $B$ be two graphs.
By standard definition $A$ is a subgraph of $B$ if $V(A)\subseteq V(B)$ and $E(A) \subseteq E(B)$. Moreover, if two vertices in $A$ are adjacent whenever they are adjacent in $B$, then $A$ is said to be an induced subgraph of $B$.
	
If $A$ is isomorphic to a subgraph (not necessarity induced) of $B$, then $B$ is said to contain $A$. 

Let $H$ be a graph on $k$ vertices, and let $G$ be a graph on at least $k$ vertices. Then $G$ is said to be $H$-exact if any induced subgraph of $G$ on $k$ vertices is either isomorphic to $H$ or it neither contains nor is contained in $H$. Then, Exact$(H)$ is the family of all graphs $G$ which are $H$-exact.

We say that  a graph on $k$ vertices is incomparable with $H$ if it neither contains nor is contained in $H$. The set of all such graphs is denoted by Inc$(H)$. Then we let Leg$(H)$ be $H\cup \mbox{Inc}(H)$. Therefore $G$ is $H$-exact if all its induced subgraphs on $k$ vertices are in Leg$(H)$.  On the other hand if a graph on $k$ vertices contains or is contained in $H$, we say it is comparable to $H$ and the set of graphs comparable to $H$ (including $H$) is denoted by Com($H$).  We let Com$^*(H)= \text{Com}(H) \backslash \{H\}$.

We now present the two main definitions:\\
 
For a graph $H$ on $k$ vertices, let 
\begin{enumerate}
\item{$f(H) = \max\{ n : \exists G \in \text{Exact}(H), |V(G)| = n \}$.}
\item{$g(H) =$ \[\min\{ n : \forall G, |V(G)| \geq n, G\text{ contains a } k \text{-vertex induced graph } F \in \text{Com}(H) \}.\]}
\end{enumerate}

Our main concern in this paper  is either to determine the exact value of $f(H)$ and $g(H)$, or to find lower and upper bounds for them.


In Section 2 we present some simple constructions of $H$-exact graphs, which further motivate the study of $f(H)$, as well as some simple facts concerning the notation.

In Section 3   we consider first problem 2 concerning the function $g(H)$ and prove  Theorem 3.1 that  $g(H) = R(H,\overline{H}) = R(\overline{H},H) = g(\overline{H})$ and use this result in Theorem 3.3 to get the lower bound $ f(H) = f(\overline{H})  \geq R(H,\overline{H}) - 1$. 

This result has the immediate consequence that every graph $G$ on at least $R(H,\overline{H})$ vertices contains a $k$-vertex induced subgraph $F$ comparable with $H$, and a $k$-vertex induced subgraph $Q$ comparable with $\overline{H}$, which supplies a strong answer to the phenomenon mentioned in the start of the introduction.  A consequence of this lower bound is that  $f(H)$ can grow exponentially with $|V (H)|$, despite the difficulty to supply constructive examples of more than quadratic growth.

We then consider $H = K_k \setminus Q$, and in particular prove that there exist $0 < c_0 < c_1$  such that $\frac{c_0k^2}{\log k} \leq  f(K_k\setminus K_3) \leq \frac{c_1k^2}{\log k}$.

In Section 4 we determine the exact value of $f(H)$ in terms of the Ramsey number involved  for two families of graphs.

For a given graph $H$ on $k$ vertices ($H \notin \{K_k,E_k\}$), let \[U(H) = \{G : |V (G)| = k, \, e(G) = e(H) + 1, \, G \mbox{ contains } H\}\]  and let \[U(\overline{H}) = \{G : |V (G)| = k, \, e(G) = e(\overline{H}) + 1, 6, \, G \mbox{ contains } \overline{H}\}\]

We prove Theorem 4.2:
\noindent \textit{For $H$ a graph on $k$ vertices, $H \notin \{K_k,E_k\}$, $f(H) = R(U(H), U(\overline{H})) -1 \leq R(K_k,K_k) -1$. Otherwise,  $f(H) =\infty$}.

However, in practice it is not easy to extract lower and upper bounds from this Ramsey type expression, and we give more specific bounds when the structure of $H$ is restricted.  

In Section 5 we consider lower bounds for $f(H)$ in terms of graph parameters (mainly the independence number and the matching number), in the spirit of classical lower bounds in Ramsey theory.   In Theorem 5.1, we prove, among other results,  that if $H$ is connected  on $k$ vertices with independence number $\alpha(H)$, then $f(H) \geq (\alpha(H) - 1)(k - 1)$ and in Theorem 5.3,  we prove that if $H$ is a connected triangle-free graph on $k$ vertices and matching number $\mu(H)$, then $f(H) \geq (k - \mu(H) - 1)(k - 1)$. 

Section 6 is devoted to trees and matchings. Here we prove the results mentioned in the abstract:

\begin{enumerate}\item{For a tree $T$ on $k\geq 3$ vertices, $ (k - 1)(\left \lceil \frac{k}{2} \right  \rceil -1 ) \leq   f(T) \leq ( k-1)^2$.}
\item{For $k \geq 4$, $f(K_{1,k-1}) = (k -1)(k-2)$.}
\item{For $k \geq 5$, $f(P_k)=\frac{(k-1)^2}{2}$ if $k$ is odd and $f(P_k)=\frac{(k-1)(k-2)}{2}+1$ if $k$ is even.}
\item{$ f(nK_2) = 3n$ for $n = 2, 3$ and $f(nK_2) = 4n - 4$ for $n \geq 4$.}
\end{enumerate}

Section 7 contains open  problems concerning $f(H)$.

\section{Basic Observations and Simple Examples}

We first give some simple examples to illustrate the concept of $H$-exact.
\medskip

\noindent \textbf{Construction via vertex transitive graphs}:  Let $F$ be a vertex transitive graph on $k$ vertices and let $G = tF$ (the union of $t \ge 1$ vertex disjoint copies of $F$). Let $F^*=F \setminus \{v\}$, which is independent of $v$ since $F$ is vertex transitive. 

Deleting any vertex from $G = tF$, we get $H= F^* \cup (t-1)F$ and hence $G$ is $H$-exact. Now $\overline{G} = \overline{tF}$ (the complement of $G$) is also a vertex transitive graph and is $Q$-exact where $Q = \overline {G} \setminus \{v\}$.
 \medskip

\noindent \textbf{Construction via regular graphs}:  For every $r$-regular graph $G$ on $k+1$ vertices and any subgraph $H$ on $k$ vertices, $G$ is $H$-exact. This is because for every $k$-vertex induced subgraph $H$, the number of edges is fixed at $e(H)=e(G) - r$. Thus, either $F$, an induced subgraph of $G$ on $k$ vertices, is isomorphic to $H$ or $F$ is incomparable with $H$, hence $G$ is $H$-exact. 
\medskip

So, finding examples of $H$-exact graphs where $|V(G)| = |V(H)| +1$ is easy.  The situation becomes more involved if $|V(G)| \geq |V(H)| +2$, since we know that if in $G$ every $k$ vertices induce a subgraph having the same number of edges, then either $|V(G)| = k + 1$ and $G$ is regular, or $|V(G)| \geq k + 2$ and $ G \in \{K_n,E_n\}$, as shown in \cite{Caro1994}.  Hence the problem of the existence/construction of $H$-exact graphs with $|V(G)| \geq |V(H)| +2$ as well as the determination or bounding $f(H)$ is of interest, and we shall deal with these problems in the following sections.

\bigskip

We now give some useful observations which we shall use in the sequel.

\bigskip

\noindent
\textbf{Fact 2.1.}

\begin{enumerate}

\item
For $H \in \{ K_k , K_k \setminus \{e\} , E_k , K_2 \cup (k-2)K_1 \}$ we have $\Leg(H) = \{H\}$.

\item
$\Leg(\overline{H}) = \{ \overline{F} : F \in \Leg(H) \}$.

\item
Given $H$ on $k$ vertices, every $F$ on $k$ vertices with $e(F) = e(H)$ is either equal to  or incomparable with $H$, hence $F \in \Leg(H)$.

\item
If $G$ is not $H$-exact and $G$ is an induced subgraph of $G^*$, then $G^*$ is not $H$-exact. In particular if there is no $H$-exact graph on $n_o(H)$ vertices then there is no $H$-exact graph on $n_0(H) \geq |H|+1$ vertices.

 \item $G$ is $H$-exact if and only if $\overline{G}$ is $\overline{H}$-exact.
\item
$f(H) = f(\overline{H})$.  This because any $k$-vertex induced subgraph
$F$ in $G$ corresponds to a $k$-vertex induced subgraph $\overline{F}$ in $\overline{G}$ and either $F = H$ and then $\overline{F}=\overline{H}$ or $F$ is incomparable with $H$ which implies that $\overline{F}$ is incomparable with $\overline{H}$. 
\item{We now list the exact values of $f(H)$ and $g(H)$ when $2 \leq |V(H)| \leq 4$.  We use the fact that $f(H)=f(\overline{H})$ (Observation 6 above) and $g(H) = R(H, \overline{H}) = R(\overline{H}, H) = g(\overline{H})$ (Theorem 3.1). The values for small Ramsey numbers used can be found in \cite{Chvtal1972GeneralizedRT}
\begin{center}
\begin{tabular}{ |p{2cm}|p{2cm}|p{2cm}|p{3cm}| p{3cm}| }
 \hline
&$H$& $\overline{H}$ &$f(H)=f(\overline{H})$& $g(H)=g(\overline{H})$\\
\hline 
 $\mathbf{n=2}$& & & &\\
\hline
&$K_2$ &$E_2$ &$\infty$&2\\
 \hline 
 $\mathbf{n=3}$& & & &\\
\hline
&$K_3$ &$E_3$ &$\infty$&3\\
&$P_3$&$K_1 \cup K_2$&4 &3\\
\hline 
 $\mathbf{n=4}$& & & &\\
\hline
&$K_4$ &$E_4$ &$\infty$&4\\
&$K_4 - \{e\}$ &$2K_1 \cup K_2$ &$4$&4\\
&$C_4$ &$2K_2$ &$6$&5\\
&$P_3 \cup K_1$ &$\text{Paw}$&4&5\\
&$K_{1,3}$ &$K_3 \cup K_1$ &6&7\\
&$P_4$&$P_4$&5 &5\\
\hline 
\end{tabular}
	\end{center}}
\end{enumerate}

\section{The function $g(H)$ and lower bounds for $f(H)$}
\subsection{The function $g(H)$}
Recall that for a graph $H$ on $k$ vertices, \[g(H) = \min\{n : \forall G, |V (G)|\geq n,\mbox{ contains a $k$-vertex induced graph $F \in \text{Com}(H)$}\}.\] The next theorem relates $g(H)$ to $R(H, \overline{H})$.

\begin{theorem} \label{gH}
$g(H) = R(H, \overline{H}) = R(\overline{H}, H) = g(\overline{H})$.
\end{theorem}
\begin{proof}
\textit{Upper bound:} Suppose $|V(G)| \ge R(H, \overline{H})$. 
By the definition of Ramsey numbers, either $G$ contains $H$ or $\overline{G}$ contains $\overline{H}$ (not necessarily induced). If $G$ contains $H$ (even not as an induced subgraph), then the induced graph $F$ on this vertex set contains $H$ and hence $F \in \text{Com}(H)$. Otherwise, $\overline{G}$ contains $\overline{H}$. If $\overline{H}$ is induced in the $k$-vertex set, then in $G$ this same $k$-vertex set induces $H$. 
Else if $\overline{H}$ is not induced, then this vertex set induces, in $\overline{G}$, some $\overline{F}$ containing $\overline{H}$; and the same vertex set induces, in $G$, the subgraph ${F}$ contained in $H$. Hence $g(H) \leq R(H,\overline{H})$.

\textit{Lower bound:} Consider the largest $n$ such that $E(K_n)$ has a red-blue colouring without monochromatic red $H$ or monochromatic blue $\overline{H}$. Let $G$ be the red graph, then $|V(G)| = R(H, \overline{H}) - 1$ and $G$ cannot contain $H$ or a $k$-vertex graph containing $H$. If it contains an induced $k$-vertex $F$ contained in $H$, then in $\overline{G}$, $\overline{F}$ contains $\overline{H}$, a contradiction. Hence $g(H) = R(H, \overline{H})$.

By symmetry of Ramsey numbers, $g(H) = R(H, \overline{H}) = R(\overline{H}, H) = g(\overline{H})$ follows by the same argument.
\end{proof}
\begin{corollary}Every graph $G$ on at least $R(H, \overline{H})$ vertices contains a $k$-vertex induced subgraph $Q \in \text{Com}(H)$ and a $k$-vertex induced subgraph $F\in \text{Com}(\overline{H})$.
\end{corollary}

\subsection{Lower bounds for $f(H)$}
The following result supplies a lower bound for $f(H)$.
\begin{theorem} \label{minfH}
$f(H)=f(\overline{H}) \geq R(H,\overline{H}) - 1$

\end{theorem}

\begin{proof}
For $H \in \{  K_k, K_k\backslash e, E_k, \overline{K_k\backslash e}\}$, this is trivial by direct checking,  since in these cases $R(H,\overline{H}) - 1=|V(H)|-1= k-1$ while $f(H)\geq k$ for any graph $H$ on $k$ vertices.

For any other $H$ consider a stronger version of Exact($H$), with only Inc($H$) =  Leg($H$) $\backslash  \{H\}$, yielding the family Exact$^*$($H$) of graphs $G$ in which every induced subgraph on $k$ vertices is incomparable with  $H$.
We define $h(H)$ as the maximum order $n=|V(G)|$ of graphs $G \in \text{Exact}^*(H)$.
(Since $K_k$, $K_k\backslash e$ and their complements are comparable with every graph on $k$ vertices,  Exact$^*$($H$) is empty for them, with undefined $h(H)$.)


Clearly $f(H) \geq h(H)$, moreover $h(H) = h(\overline{H})$ since $F$ is incomparable with $H$ if and only if $\overline{F}$ is incomparable with $\overline{H}$ (see Observation 2).  We also note that $n' = h(H)+1$ is the smallest number such that every graph on at least $n$ vertices contains a $k$-vertex induced subgraph comparable with $H$.

Hence by the definition of $g(H)$ and Theorem \ref{gH} we get $h(H) = g(H) - 1 = R(H,\overline{H}) -1$, and consequently

 \[f(\overline{H}) =f(H) \geq h(H) \geq R(H,\overline{H})  - 1  = R(\overline{H},H) - 1.\]
\end{proof}

\noindent \textbf{Remark}: The Ramsey number $R(H,\overline{H})$, and hence also $f(H)=f(\overline{H})$, can be exponentially large in terms of $|V (H)|$, despite the difficulty to construct explicit examples, as demonstrated by $H =  K_k  \cup E_k$  and $\overline{H} = K_k + E_k$.  Then $|V(H)| = 2k$ and  \[R(H,\overline{H}) > R(k,k) > k2^{k/2} \left(\frac{\sqrt{2}}{e} +o(1) \right)  = \frac{|V(H)|}{2}2^{|V(H)|/4} \left(\frac{\sqrt{2}}{e} +o(1) \right)\]
by Spencer's lower bound in \cite{SPENCER1975108}.

\subsection{The case $H=K_k \setminus Q$}
The following result is proven in \cite{Caro1994}:

\begin{theorem} \label{regular}
If in $G$ every set of $k$ vertices induces a subgraph with the same  number of edges, then either $|V(G)| = k + 1$ and $G$ is regular, or $|V(G)| \geq k + 2$  and $G \in \{K_n,E_n\}$.
\end{theorem}
\begin{proposition}
For $H = K_k \backslash \{e\}$, $k \geq 3$, we have $f(H) = k + 1$ if $k = 3$, and $f(H)=k$ otherwise.
\end{proposition}
\begin{proof}
Of course, $f(k)\geq k$ holds for every $H$.
Consider $H = K_k \backslash \{e\}$, $k \geq 3$. Then Leg$(H) = \{H\}$.  If $G$ on at least $k + 2$  vertices is $H$-exact, then every $k$-subset of $G$ either induces $H$ or  is incomparable to $H$. But since Leg$(H) = \{H\}$, the second alternative is impossible.  Hence, every $k$-subset induces $H$ and must have the same number of edges.  By Theorem \ref{regular}, $G$ should be $K_n$ or $E_n$, both of which are comparable with $H$  and are members of Com$^*(H)$.
Thus for $H = K_k \backslash \{e\}$ we obtain $k \leq f(H) < k+2$.

Now consider $G$  which is $H$-exact of order $k +1$. As above, $G$ must be an $r$-regular graph on $k + 1$ vertices.  If $r = k$, then $G=K_{k+1}$ and every $k$-vertex set is $K_k \in \text{Com}^*(H)$.  

So let us assume that $r \leq k - 1$. Observe that $e(G) = \frac{(k+1)r}{2}$ and $e(G - v) = \frac{(k+1)r}{2}  - r = \frac{(k-1)r}{2}$.  Equating  $\frac{(k-1)r}{2}  = \frac{k(k-1)}{2} -1$ simplifies to   $2 = (k - 1)(k - r)$. This is  possible only if $k = 3$ and $r = 2$, which corresponds to the case of $P_3 = K_3 \backslash \{e\}$ with $G =  C_4$.
\end{proof}

\begin{lemma} \label{indlemma}
$\mbox{  }$
\begin{enumerate}
\item{If $H$ has no isolated vertices then $R(K_k , H)  \geq k + |V(H)|  - 2$.}
\item{For $Q$ a graph without isolated vertices  on $q \geq 3$ vertices and  $k \geq  q +2  \geq  5$,  $R(K_k ,Q) \geq  k +q  - \alpha(Q)$.}
\end{enumerate}
\end{lemma}
\begin{proof}
$\mbox{  }$
\begin{enumerate}
\item{Take a set of $k + |V(H)| - 3$ vertices and partition them into two sets: $A$ with $k-2$ vertices and $B$  with $|V(H)|-1$ vertices.  Colour all edges within $B$ blue forming a blue $K_{|V(H)| - 1}$, and all other edges red.  Since $H$ has no isolated vertices, the blue graph contains only $|V(H)| - 1$ vertices and hence no blue copy of $H$. The largest red clique contains all vertices from $A$ and just one vertex from $B$ and therefore has no red $K_k$.  Hence $R(K_k , H) \geq k +|V(H)| - 2$.}
\item{Suppose now $Q$ is not a complete graph  so  $\alpha(Q) \geq 2$.  Then  by part 1 above, $R(K_k ,Q) \geq  k +q - 2 \geq k +q - \alpha(Q)$.  If $Q$ is a complete graph on at least three vertices,  then $R(K_k ,Q)  = R(k,q) > k + q - 2$  always.}
\end{enumerate}
\end{proof}

\begin{theorem}
$\mbox{ }$
\begin{enumerate}
\item{Let $Q$ be graph on $q \ge 3$ vertices with no isolated vertices and with independence number $\alpha(Q)$. Suppose $k \ge q + 2$. Let $H = K_k \setminus Q$, then \[R(K_{k-q+\alpha(Q)}, Q) - 1 \le f(H) \le R(K_k, Q) - 1.\]}
\item{Let $Q \notin \{K_q, E_q\}$ be a graph on $q \ge 3$ vertices with independence number $\alpha(Q)$ and $k = q+1$. Let $H = K_k \setminus Q$, then \[R(K_{k-q+\alpha(Q)}, Q) - 1 \le f(H) \le \max\{R(K_k, Q) - 1, f(Q) + q\}.\]}
\end{enumerate}
\end{theorem}
 
\begin{proof}
$\mbox{ }$
\begin{enumerate}
\item{\textit{Lower bound}:  Observe first that if $|V(Q)| = q \le k$ then $K_k \setminus Q$ contains $K_{k-q+\alpha(Q)}$.
 
By Theorem \ref{minfH} and this observation, $f(H) \ge R(K_k \setminus Q, \overline{(K_k \setminus Q)}) - 1 \ge R(K_{k-q+\alpha(Q)}, Q \cup (k-q)K_1) - 1 \ge R(K_{k-q+\alpha(Q)}, Q) - 1$.

\noindent \textit{Upper bound}:  Let $G$ be an $H$-exact graph on $n \ge \max\{R(K_k, Q), k+q-\alpha(Q)\}$ vertices, where $H = K_k \setminus Q$, and recall that $G$ is $H$-exact if and only if $\overline{G}$ is $\overline{H}$-exact.  
Consider $\overline{H} = Q \cup (k-q)K_1$. By assumption on $n$, either $G$ contains $K_k$ and we are done, or $\overline{G}$ contains $Q$.  
So, assume $\overline{G}$ contains $Q$.
Observe that if there is a further edge in $\overline{G}$, no matter if it is incident to two vertices of $Q$, or one vertex of $Q$, or vertex-disjoint from $Q$, then since $k \ge q + 2$ we can complete $Q$ to an induced subgraph $F$ on $k$ vertices strictly containing $\overline{H} = Q \cup (k-q)K_1$.
Hence in that case $\overline{G}$ would not be $\overline{H}$-exact and $G$ would not be $H$-exact, contrary to the choice of $G$.  
Consequently a copy of $Q$ in $\overline{G}$ must have the property that it is induced and there are no other edges between $V \setminus Q$ and $Q$, as well as no other edges in $V \setminus Q$.
So, $(V \setminus Q) \cup B$, where $B$ is a maximum independent set in $Q$, is an independent set of order $n - q + \alpha(Q)$\,; and if this is at least $k$, we have an independent set of size $k$ in $\overline{G}$. Hence $K_k \in \text{Com}^*(K_k \setminus Q)$ in $G$, a contradiction.
Thus for $k \ge |Q| + 2$, $f(K_k \setminus Q) \le \max\{R(K_k, Q), k+q-\alpha(Q)\} - 1$.  However  $\max\{R(K_k,Q), k+q-\alpha(Q)\}-1 = R(K_k , Q) - 1$ by Lemma \ref{indlemma}.}       
 
\item{The lower bound is as above.

\noindent \textit{Upper bound}:
We first note that the reason to exclude $Q \in \{ K_q , E_q \}$  is that in both cases $f(Q) = \infty$, hence the upper bound is useless.

  Let $G$ be an $H$-exact graph on $n \ge \max\{R(K_k, Q)-1, f(Q) + q\}+1$ vertices, where $H = K_k \setminus Q$, $k = q+1$. This $H$ is in fact $\overline{Q} + K_1$.
Consider $\overline{H} = Q \cup K_1$. Since $n \geq R(K_k,Q)$, then either $G$ contains $K_k$ and we are done, or $\overline{G}$ contains $Q$.  
So, assume $\overline{G}$ contains $Q$.
Similarly as above, if we can add an edge, which is either incident with two vertices of $Q$, or to one vertex of $Q$, then since $k = q+1$ we can complete $Q$ to an induced subgraph $F$ on $k$ vertices strictly containing $\overline{H}$.
Hence the contradiction is obtained that $\overline{G}$ is not $\overline{H}$-exact and $G$ is not $H$-exact.  
Consequently a copy of $Q$ in $\overline{G}$ must have the property that it is induced and there are no other edges between $V \setminus Q$ and $Q$.
So $\langle V \setminus Q, Q \rangle$ contains no edges in $\overline{G}$. Hence in $G$, $\langle V \setminus Q, Q \rangle$ is $K_{n-q, q}$.
But now in $G$  every $q$-subset in the $n-q$ side must induce either $\overline{Q}$  or a graph incomparable with $\overline{Q}$ (since $H=\overline{Q}+K_1$), otherwise with any  vertex of the $q$-side we have an induced  $k$-vertex subgraph $F\in \text{Com}^*(H)$ (observe that this argument applies since $Q \notin \{K_q, E_q\}$).  Hence, in  $G$, the induced subgraph $R$  on this $n - q$ side, is $Q$-exact and $n - q = |V(R)|  \leq f(Q)$.  This implies $n \leq f(Q) + q$, contradicting the assumption that $n \geq f(Q) + q+1$. }
\end{enumerate}
\end{proof}

\noindent \textbf{Example}: Suppose $Q = K_3$. Then using part  1  and classical bounds \cite{campos2025new,kim1995ramsey} on $R(K_k,K_3)$, we get \[\frac{c_0k^2}{\log k} \le  f(K_k \setminus K_3) \le \frac{c_1k^2}{\log k}.\]  
 
 

\section{General upper bounds for $f(H)$ via Ramsey numbers}

We now develop several upper bounds on $f(H)$ in terms of related Ramsey numbers.

\begin{definition}
For a given graph $H$ on $k$ vertices ($H \notin \{K_k, E_k\}$), let \[U(H) = \{ G: |V(G)|=k, \, e(G)=e(H) +1, \, G \text{ contains } H \}\] and \[U(\overline{H}) = \{ G: |V(G)|=k, \, e(G)=e(\overline{H}) +1, \, G \text{ contains } \overline{H} \}\,.\]
\end{definition}

Clearly $H$ is not in $U(H)$ and $\overline{H}$ is not in $U(\overline{H})$.
But $H$ is contained in every member of $U(H)$ and $\overline{H}$ is contained in every member of $U(\overline{H})$\,; moreover both $U(H)$ and $U(\overline{H})$ are not empty as $H \notin \{K_k,E_k\}$.  

\begin{theorem}
 For $H$ a graph on $k$ vertices, $H \notin \{K_k,E_k\}$, $f(H) = R(U(H) , U(\overline{H}))-1 \leq R(K_k,K_k)-1$.  Otherwise $f(H)=\infty$.
\end{theorem}
\begin{proof}
\textit{Upper bound:} Suppose $H \notin \{K_k,E_k\}$.  Consider the Ramsey number $R(U(H) , U(\overline{H}) )$, that is the minimum $n$ such that in every red-blue colouring of $K_n$ either there is a red member of $U(H)$ or a blue member of $U(\overline{H})$ (equivalently either $G$, the graph on $n$ vertices induced by the red edges, contains a member of $U(H)$ or $\overline{G}$ contains a member of $U(\overline{H})$).  If $G$ contains a member of $U(H)$ then it is not exact by the definition of $U(H)$. If $\overline{G}$ contains a member $F$ of $U(\overline{H})$  then the induced subgraph $Q$ (could be $F$) on the vertex set of $F$ is a graph containing $\overline{H}$ but $Q \neq \overline{H}$.  Hence in $G$ the same vertex set induces a subgraph  $\overline{Q}$ contained in $H$  but $\overline{Q} \neq H$. Hence $G$ is not $H$-exact

Observe this Ramsey number is well -defined as both classes are not empty and $K_k$ contains all members of these families (or is its only member), and in particular $f(H) < R(U(H) , U(\overline{H}) ) \leq R(K_k , K_k) $.

\medskip

\noindent
\textit{Lower bound:} Suppose now $n = R(U(H), U(\overline{H}))- 1$.  Then by definition of Ramsey numbers, there is a red-blue colouring of $E(K_n)$ with no red member (not necessarily induced)  of $U(H)$ and no blue member (not necessarily induced) of $U(\overline{H})$.
Equivalently, calling the graph of red edges $G$, neither $G$ contains a member of $U(H)$ nor $\overline{G}$ contains a member of $U(\overline{H})$.

Consider any $k$-vertex induced subgraph $F$ of $G$.
It cannot strictly contain $H$, otherwise a member of $U(H)$ would necessarily occur in $G$. Hence $F$ either has at least $e(H) +1$ edges and is incomparable with $H$, or has exactly $e(H)$ edges and therefore is either an induced copy of $H$ or is incomparable with $H$, or else it contains at most $e(H) - 1$ edges.
We claim that in the latter case $F$ cannot be comparable with $H$;
equivalently, as $e(F)<e(H)$, it cannot be a subgraph of $H$.
%
%
Indeed, otherwise in $\overline{G}$, its complement $\overline{F}$ induced on the same vertex set $V(F)$ would contain at least $e(\overline{H}) +1$ edges and would be comparable with $\overline{H}$.
Then the blue subgraph of $K_n$ would contain a member of $U(\overline{H})$, a contradiction to the choice of the red-blue colouring we started with.

In summary, every $k$-vertex subgraph $F$ in $G$ either induces $H$ or is incomparable with $H$, meaning $G$ is $H$-exact, and similarly  $\overline{G}$ is $\overline{H}$-exact.

Lastly we observe $\text{Exact}(K_k)=\{K_n: n\geq k\}$ and $\text{Exact}(E_k)=\{E_n: n\geq k\}$ and hence in both cases $f(H)= \infty$.
\end{proof}
The next corollary sometimes gives an effective upper bound on $f(H)$. 

\begin{corollary} \label{ub_ramsey}
Let $H \notin \{K_k,E_k\}$ be a graph on $k$  vertices and $F$  a graph on $k$ vertices containing $H$. Then $f(H) < R(F, K_k)$. In particular, for every acyclic graph $H$ on $k$ vertices and at most $k-2$ edges, $f(H) \leq (k-1)^2$.
\end{corollary}
\begin{proof}

If $G$ is a graph on $n \geq R(F, K_k)$ vertices, then  $G$ either contains induced $E_k \in \text{Com}^*(H)$ or it contains  a (not necessarily induced) copy of $F  \in \text{Com}^*(H)$.  Hence $f(H) < R(F, K_k)$.  If $H$ is acyclic on $k$ vertices and at most $k-2$ edges, then it is a forest  that is contained in some tree $T$ on $k$ vertices hence $f(H) \leq R(T,K_k) - 1 = (k-1)^2$ by Chv\'atal's theorem \cite{Chvtal1977TreecompleteGR}.  
\end{proof}

The next theorem is useful when a graph has a cut-vertex.

\begin{theorem}\label{impthm}
Let $F$ and $H$ be two non-empty graphs each with at least one edge.  Let $w$ be a vertex in $F$ and $v$ a vertex in $H$.  Let $Q$ be the graph obtained by identifying $w$ and $v$ into one vertex and let $Q(e)=Q-e$ for some edge $e$ of $Q$.  Then \[f(Q(e))  <  \min \{ R(H,K_{|Q|} ) +  |H|(R(F,K_{|Q|}) - 1) , R(F,K_{|Q|} ) +  |F|( R(H,K_{|Q|} )- 1)  \}.\]
\end{theorem}

 \begin{proof}

We tackle each upper bound separately.
\begin{enumerate}
\item{$f(Q(e))  <  R(H,K_{|Q|} ) +  |H|(R(F,K_{|Q|}) - 1)$\\

Consider a graph $G$ on $R(H,K_{|Q|} ) +  |H|(R(F,K_{|Q|}) - 1)$ vertices.  Then $G$ contains  $H$ or induced $E_{|Q|}$.  In the latter case, $E_{|Q|}$ is comparable with $Q( e)$,  because $Q(e )$ contains at least one edge, that is $E_{|Q|} \in \text{Com}^*(Q(e))$.
 
So we assume $G$ contains a copy of $H$ (not necessarily induced).  Delete the vertices of this $H$. We still have enough vertices to have another copy of $H$ which we delete, and continue deleting any remaining copy of $H$.  We can delete $R(F , K_{|Q|})$  vertex-disjoint copies of $H$, because if we delete fewer, we still have at least $R(H, K_{|Q|})$ vertices.  

Let these $R(F , K_{|Q|})$ vertex-disjoint copies of $H$ be labelled  $H(j) = (a_{j,1}, a_{j,2},\ldots,a_{j,|H|}): j = 1 \dots R(F , K_{|Q|})$ , where $a­_ {­­j,1}$ is equivalent to vertex $v$ in the $j^{th}$ copy of $H$ 

Consider the set $S$  of $R(F , K_{|Q|})$  vertices formed by $S = \{ a_{j,1} , j = 1,\ldots, R(F , K_{|Q|})   \}$.  As before, either there is induced $E_{|Q|}$ and we are done or there  is a copy of $F$ in $S$.  So, consider  the vertex $w$ in this copy of $F$ --- clearly $w=a_{k,1} \in S$, and recall these each have the role of $v$ so $w = v \in H(k)$.

We take $F$ and the copy of $H(k)$ (to which $a_{k,1}$ belongs) --- this gives a copy of $Q$ (not necessarily induced) strictly containing $Q(e)$. }
\item{$f(Q(e))  <  R(F,K_{|Q|} ) +  |F|( R(H,K_{|Q|} )- 1) $\\

Consider  a graph $G$ on  $R(F,K_{|Q|} ) +  |F|( R(H,K_{|Q|} )- 1)$  vertices. 

 Then, again, either $G$ contains $F$ or induced $E_{|Q|}$.  In the latter case, $E_{|Q|}$ is comparable with $Q( e)$,  because $Q(e )$ contains at least one edge, that is $E_{|Q|} \in \text{Com}^*(Q(e))$.

So we assume there is a copy of $F$ in $G$ (not necessarily induced).  Delete the vertices of this $F$. We still have enough vertices to have another copy of $F$ which we delete, and continue deleting any remaining copy of $F$.  We can delete $R(H , K_{|Q|})$  vertex-disjoint copies of $F$, because if we delete fewer, we still have at least $R(F, K_{|Q|})$ vertices.  

Let these $R(H , K_{|Q|})$ vertex-disjoint copies of $F$ be labelled  $F(j) = (a_{j,1}, a_{j,2},\ldots,a_{j,|F|}): j = 1 \dots R(H , K_{|Q|})$, where $a­_ {­­j,1}$ is equivalent to vertex $w$ in the $j^{th}$ copy of $F$.

Consider the set $S$  of $R(H , K_{|Q|})$  vertices formed by $S = \{ a_{j,1} , j = 1,\ldots, R(H , K_{|Q|})   \}$.  As before, either there is induced $E_{|Q|}$ and we are done or there  is a copy of $H$ in $S$.  So, consider  the vertex $v$ of $H$.  Clearly $v$ must coincide with some $a_{k,1}$ so $v = w$.

We take $H$ and the copy of $F(k)$ (to which $a_{k,1}$ belongs) --- this again gives a copy of $Q$ (not necessarily induced) strictly containing $Q(e)$. }
\end{enumerate}
The combination of the two cases yields
 \[f(Q(e))  <  \min \{ R(H,K_{|Q|} ) +  |H|(R(F,K_{|Q|}) - 1) , R(F,K_{|Q|} ) +  |F|( R(H,K_{|Q|} )- 1)  \}\]
as claimed.
\end{proof}
 
 The next theorem shows that in certain cases a stronger version of Theorem \ref{impthm} can be obtained.

\begin{theorem}

Suppose $H$ is a triangle-free graph on $k$ vertices,  containing a vertex $w$ adjacent to at least two leaves $u$ and $v$.  Then $f(H)  \leq R(K_3,K_k) +R(H,K_k)  - 1 $.
\end{theorem}
 
\begin{proof}
Consider a red-blue colouring of $E(K_n)$  where $n \geq  R(K_3 , K_k)  +  R(H, K_k)$.  Using the $R(K_3,K_k)$ part there is either a red $K_3$ or a blue $K_k$. If there is a blue $K_k$ the graph $G$ induced by the red edges contains an induced  $E_k \in \text{Com}^*(H)$.  Hence,  we assume there is a red $K_3$.  We can delete $ \left \lfloor \frac{R(H,K_k)}{3} \right \rfloor  +1$ disjoint copies of $K_3$'s, since deleting at most $ \left \lfloor \frac{R(H,K_k)}{3} \right \rfloor$ such copies, only at most $R(H,K_k)$ vertices would be deleted, still leaving at least $R(K_3,K_k)$ vertices and we can delete one more red $K_3$.

These $t = \left \lfloor \frac{R(H,K_k)}{3} \right \rfloor+1$ vertex disjoint red triangles  span at least  $R(H,K_k)$  vertices.


Now consider the graph spanned by the vertices of these $t$ triangles.  As there are at least $R(H,K_k)$ vertices, by Ramsey we get either a red $H$ or a blue $K_k$ (already excluded).  So,  consider this  red copy of $H$.  If $H$ contains three vertices in one of the red triangles, then it is a red copy of $H$ containing a red triangle and we are done  since $H$ is triangle-free; so we assume there is no such red triangle in $H$.

Recall that $H$ contains a vertex $w$ (adjacent to two  leaves $u$ and $v$). Suppose $w$ is in a red triangle $A = \{ w ,x ,y \}$. If none of $x$ and $y$ are occupied by vertices of  $H$, then $u$ and $v$ are other vertices, and we can replace them by $x$ and $y$ --- this causes no change to the connectivity  (or components)  of the red $H$, due to the role of $\{ w ,v ,u \}$, and we have a red $H$  containing a red triangle.

If, without loss of generality, $x$ is occupied by $v$,  then  $y$ is not occupied by any other vertex of $H$ otherwise there is a non-induced red copy of $H$ containing a red triangle. 
But then $u$ can be moved to $y$ to again give a red $K_3$ in a red  $H$.    

If, without loss of generality, $x$ is occupied by a vertex $z$ of $H$  which is not $u$ or $v$, then $y$ is not occupied by any other vertex of $H$, otherwise there is a red triangle in  a red $H$.
But then say $u$ is in another place and we move $u$ to $y$ to get a red $H$ with a red triangle.

This covers all cases and hence $f(H) \leq R(K_3,K_k) + R(H,K_k)-1$.  
\end{proof}

\noindent \textbf{Remark}:   For a tree $T$ on $k$ vertices having a vertex $w$ adjacent to at least two leaves, $ f(T) \leq R(K_3,K_k) + (k-1)^2$ holds,  because we know that $R(T,K_k) = ( k-1)^2 +1$ by Chv\'atal's Theorem \cite{Chvtal1977TreecompleteGR}.  However in Section 6 we will show via an ad-hoc argument tailored for trees  that for every tree $T$ on $k \geq 3$ vertices, the upper bound $ f(T) \leq ( k-1)^2$ is valid.

\section{Lower bounds for $f(H)$ via graph parameters}
\label{s:LB}
This section concerns useful lower bounds for $f(H)$ in terms of some important graph parameters.

\begin{theorem}\label{treethree}
Let $H$ be a graph on $k\geq 3$ vertices with independence number $\alpha(H)$.  Then 
\begin{enumerate}
\item{$f(H) \geq R(H,K_{\alpha(H)}) - 1$.  }
\item{If $H$ is connected, then  $f(H)  \geq (\alpha(H) - 1)(k-1)$.}
\item{If $H$ is a tree $T$ on $k$ vertices, then $f(H) \geq (\alpha(T)-1)(k-1) \geq
 (k-1)(\left \lceil \frac{k}{2} \right \rceil -1)$}.
\end{enumerate}
\end{theorem}
\begin{proof}
$\mbox{ }$
\begin{enumerate}
\item{
We first observe that  $R(H,\overline{H}) \geq R(H,K_{\alpha}(H))$.  This is because in any red-blue colouring of $E(K_n)$  that forces  either a red $H$ or a blue $\overline{H}$, either there is a red $H$ or a blue $K_{\alpha(H)}$ since $\overline{H}$ must contain $K_{\alpha(H)}$.  Hence by Theorem 3.3, $f(H) \geq R(H, K_{\alpha(H)}) -1$.  }
\item{Consider $G = (\alpha(H) - 1)K_{k-1}$. Clearly $G$ contains no copy of $H$ and hence no  induced $k$-vertex graph containing $H$ since $H$ is connected.  Observe also that $\alpha(G) = \alpha(H)-1$ and hence for any induced subgraph $G^*$ of $G$, $\alpha(G^*) \leq \alpha(G) < \alpha(H)$.  But if there is a $k$-vertex induced subgraph $F$ contained in $H$, then $\alpha(F)  \geq \alpha(H)$ which is impossible.

Hence $G$ is $H$-exact and $f(H)  \geq (\alpha(H) - 1)(k-1)$.}
\item{If $H$ is a tree $T$ on  $k$ vertices with independence number $\alpha(T)$, then by part 2, $f(H) \geq (k - 1)(\alpha(T) - 1)$. We observe that in every tree on $k$ vertices,   $\alpha(T)  \geq  \left \lceil \frac{k}{2} \right \rceil$ and the second inequality follows.}
\end{enumerate}
\end{proof}

As a matter of fact, the lower bound in part 3 of this theorem is valid also for all bipartite connected graphs of order $k$.

\begin{corollary} \label{bipart}
Let $H$  be a connected graph on $k$ vertices with matching number $\mu(H)$. Then $ f(H)  \geq ( k - 2\mu(H) - 1) (k-1)$; and if $H$ is a connected bipartite graph then $f(H)  \geq (k - \mu(H) - 1)(k-1)$.
\end{corollary}
\begin{proof}
It is well known that $\alpha(H) = k - \tau(H)$ where $\tau$ is  the size of the minimum vertex cover. Clearly $\tau(H) \leq 2 \mu(H)$, where $\mu(H)$ is the matching number of $H$ and hence $\alpha(H) \geq k - 2\mu(H)$.  Therefore, by part 2 of Theorem~\ref{treethree}, $f(H) \geq ( k - 2\mu(H) - 1) (k-1)$.

A well known Theorem by K\"onig in \cite{konig1931grafok} states that $\tau(H)=\mu(H)$ holds for bipartite graphs, hence $f(H)  \geq (k - \mu(H) - 1)(k-1)$.
\end{proof}

We now prove the stronger result that, under the connectivity requirement, the lower bound in Corollary \ref{bipart} can be extended from bipartite graphs to all triangle-free graphs.

\begin{theorem} \label{trianglefree}
Let $H$ be a connected triangle-free graph on $k$ vertices and maximum matching size $\mu(H)$. Then $f(H) \ge (k - \mu(H) - 1)(k-1)$.
\end{theorem}

\begin{proof}
Consider $G = (k - \mu(H) - 1)K_{k-1}$ (the disjoint union of $k - \mu(H) - 1$ copies of $K_{k-1}$).
Clearly, there is no copy of $H$ in $G$, as $H$ is connected.

If we take at least three vertices of the same component, we have a cycle (or a clique of size at least three), and complementing  these vertices to any $k$-vertex induced subgraph $F$, this graph contains $K_3$ but does not contain $H$ and hence $F$ is incomparabable with $H$.

If we take at most two vertices from each component, then
picking any $k$ vertices we have at least $\mu(H)+1$ induced edges from distinct components, because the number of components is only $k - \mu(H) - 1$.
In this case, however, the $k$-vertex induced subgraph $F$ has matching number $\mu(F)\geq \mu(H)+1$ while the connected $H$ is not a subgraph of the disconnected $F$, hence they are incomparable.
%
%

By the above two cases, any $k$-vertex subgraph is neither $H$ nor comparable with $H$.
Thus, $G$ is $H$-exact, forcing $f(H) \ge (k - \mu(H) - 1)(k-1)$.
\end{proof}



\section{Quadratic bounds for $f(T)$ and the exact determination of $f(H)$ for stars and matchings}

The next two theorems show that $f(T) \leq (k-1)^2$ holds for every tree $T$ on $k \geq 3$  vertices.  We first consider the case where $T$ is not a star.  
\begin{theorem}

If $T$ is a tree of order $k\geq 4$, which is not a star then $f(T) \leq (k-1)^2$.

\end{theorem}
\begin{proof}
	Let $T$ be a non-star tree of order $k$, and $G$ a $T$-exact graph of order $f(T)$.
	Assume, for a contradiction, that $|V(G)|>(k-1)^2$ holds.
	
	We first observe that $G$ contains an induced subgraph $F$ of minimum degree at least $k-1$.
	Indeed, since any $d$-degenerate graph $G$ has $\alpha(G)\geq\frac{|V(G)|}{d+1}$ (and even more strongly $\chi(G)\leq d+1$), assuming $d\leq k-2$ we would obtain $\alpha(G)  \geq \frac{|V(G)|}{k-1} > k-1$, contradicting that $\alpha(G) < k$ holds by definition.
	
	Let now $uv$ be an edge in $F$.
	If $F$ contains no 3-cycles and 4-cycles, then the set $(N_F(u) \cup N_F(v))\setminus \{u,v\}$ is independent, therefore again the contradiction $\alpha(G)\geq k$ follows by $\alpha(G)\geq \alpha(F)\geq 2k-4>k-1$ for $k\geq4$.
	Thus, the proof will be done if we show that $F$ is triangle-free and $C_4$-free.
	We prove this by demonstrating that any $C_3$ or $C_4$ can be extended to a sugraph $H$ of $F$ that contains $T$---but of course $H\cong T$ cannot hold as $H$ is not cycle-free, hence $G$ is not $T$-exact, while we have assumed that it is.

	Since $T$ is a non-star tree, it has diameter at least 3.
	Then we can build $T$ sequentially in a suitably chosen vertex order $v_1,v_2,\dots,v_k$ such that $P:=v_1v_2v_3v_4$ is a path $P_4$ in $T$, and the set $\{v_i : 1\leq i\leq j\}$ induces a (connected) subtree say $T_j$ in $T$, for every $j$ with $4\leq j\leq k$.
	Note that each $v_j$ has exactly one neighbor in $T_{j-1}$; let us denote it as $v_j^-$.
	
	Assuming that $F$ contains $C_3$ or $C_4$, embed $P$ into $F$ in such a way that the subgraph induced by $\{v_1,v_2,v_3,v_4\}$ either contains $C_4$ or is a triangle with a pendant edge, hence a non-induced $T_4$ in $F$.
	(The vertices of a 3-cycle in $F$ have degree at least $k-1>2$\,; i.e., the required pendant edge is found in an ovious way if $F$ is not triangle-free.)
	Having embedded $T_{j-1}$ for any $5\leq j\leq k$, consider the vertex $v_j^-$.
	It is of degree at least $k-1$ in $F$, but has not more than $j-2\leq k-2$ preceding embedded neighbors.
	Hence, $v_j^-$ has a neighbor in $F$ (in fact at least $k-j+1$ of them) which is a suitable choice for $v_j$.
	At the end of this process we obtain a non-induced subgraph $T_k\cong T$ in $F$.
	This contradiction to $T$-exactness proves that $F$ should be $C_3$-free and $C_4$-free, hence completing the proof of the theorem.
\end{proof}
We next determine $f( K_{1,k-1})$ exactly.
\begin{theorem}
     For $k \ge 4$, $f(K_{1,k-1}) = (k-1)(k-2)$.
\end{theorem}
\begin{proof}
We begin with the upper bound.  Let $T = K_{1,k-1}$ be a star on $k \ge 4$ vertices, and $G$ a $T$-exact graph of order $n \ge (k-1)(k-2) + 1$. We know by assumption that $\alpha(G) < k$.

Consider first $\alpha(G) = k-1$. Let $S$ be an independent set of vertices where $|S| = k-1$, and let $B = V(G) \setminus S$. 
Consider any vertex $v \in B$. If $v$ is not adjacent to all vertices of $S$, then we can embed the centre of the star $T$ at $v$ and the remaining vertices of $T$ in $S$. Since there are no edges within $S$, this forms an induced subgraph of $G$ strictly contained in $T$ (hence in $Com^*(T)$).

We may assume that every vertex of $B$ is adjacent to all vertices of $S$. Note that $|B| = n - (k-1) \ge (k-1)(k-2) + 1 - (k-1) = (k-2)^2$.  Then the graph induced by $\langle B, S \rangle$ is the complete bipartite graph $K_{k-1, |B|}$ since if $B$ contains an edge $e$, we can choose a vertex $v \in S$ and $k-1$ neighbours in $B$ containing $e$ giving an induced subgraph containing $K_{1,k-1}$. Hence, $B$ is an independent set of cardinality $|B| \ge (k-2)^2 >(k-1)$ for $k \ge 4$, which contradicts our assumption on $\alpha(G)$.

So, we may assume $\alpha(G) \le k-2$. 
Now suppose that the degeneracy of $G$ is at most $k-2$. Then $\chi(G) \le k-1$ and
\[
\alpha(G) \ge \frac{|V(G)|}{\chi(G)} \ge \frac{(k-1)(k-2) + 1}{k-1} > k-2 \,.
\]
Hence $\alpha(G) \ge k-1$, which is the case already considered. 
Consequently, the degeneracy of $G$ is at least $k-1$, meaning there exists an induced subgraph with minimum degree $\delta \ge k-1$.

If $G$ contains $K_3$, we can extend it to $K_{1,k-1}$ containing $K_3$ trivially. Otherwise, for every vertex $u$, the neighbourhood $N(u)$ is triangle-free, forming an independent set of cardinality at least $k-1$, and again we are done by the first part.

Hence $f(K_{1,k-1}) \le (k-1)(k-2)$. 

The lower bound follows from $f(T) \ge (k-1)(\alpha(T) - 1)$, which in this case is $(k-1)(k-2)$ by part 3 of Theorem \ref{treethree}.
\end{proof}

\noindent \textbf{Remark}: For $k = 3$, $K_{1,2} = P_3$ and  $f(H) = 4$ (see Table 1).

 \subsection{Paths}

We now prove that $f(P_k)$ grows asymptotically with $k^2\!/2$ as $k$ gets large.

\begin{table}[ht]
 \begin{center}
  \begin{tabular}{c|ccccccccccc}
   $k$ & 2 & 3 & 4 & 5 & 6 & 7 & 8 & 9 & 10 \\
   \hline
   $f(k)$ & $\infty$ & 4 & 5 & 8 & 11 & 18 & 22 & 32 & 37
  \end{tabular}
  \caption{$f(P_k) = (k-1)\lfloor\frac{k-1}{2}\rfloor+1-(k$ mod $2)$ for $k\geq 5$.}
 \end{center}
\end{table}

	The value of $f(P_k)$ for small $k\leq 4$ is not hard to find directly, as shown also in the table exhibited in Fact 2.1.7.
	Before determining $f(P_k)$ for $k\geq 5$, let us state the following auxiliary result.
	We denote
	
	$$
	  \epsk = \left\{
	    \begin{tabular}{ccl}
	      1 & \qquad & if \ $n \equiv 1$ \ (mod $(k-1)$) \\
	      0 && otherwise
	    \end{tabular}
	  \right.
	$$

\bsk

\begin{lemma}
\label{l:linforest}
	Let $n\geq k\geq 5$ be integers.
	Then every $P_k$-free graph of order $n$ contains an induced linear forest of order $2\lceil \frac{n}{k-1} \rceil - \epsk$.
\end{lemma}

\begin{proof}

	Let $G$ be a $P_k$-free graph on $n$ vertices, $n\geq k$.
	We choose a collection of vertex-disjoint paths $F_1,F_2,\dots$ which partition the vertex set of $G$, and satisfy the following property:
	\begin{itemize}
	 \item For every $i\geq 1$, $F_i$ is a longest path in the graph induced by $\bigcup_{j\geq i} V(F_i)$.
	\end{itemize}
	Such a collection of paths is easy to find by sequential selection, in each step taking and omitting a longest path in the remaining graph.
	Of course, $k-1 \geq |V(F_1)| \geq |V(F_2)| \geq \ldots$ holds, and $\sum_{i\geq 1} |V(F_i)| = n$.
	
	We note that if $i$ is not the largest index in the collection, and $|V(F_i)|>2$, then either the two ends of $F_i$ are nonadjacent or $V(F)$ induces a connected component of $G-\bigcup_{j<i} V(F_j)$.
	
	Now we select those $F_i$ which are single vertices, and from all the others we select the two ends of each path.
	We claim that the set $S$ obtained in this way induces a linear forest each of whose components is an edge or a vertex.
	Indeed, the singleton paths are isolated in $S$, and the (possibly adjacent) two ends of a non-singleton $F_i$ are nonadjacent to the entire $\bigcup_{j>i} V(F_j)$ due to the non-extendability of $F_i$, hence in particular nonadjacent to the ends of those paths.
	
	Graph $G$ is assumed to be $P_k$-free, therefore the number of paths is at least $\lceil \frac{n}{k-1} \rceil$.
	From each non-singleton path $F_i$ we selected $2 \geq \lceil \frac{2}{k-1}\cdot|F_i| \, \rceil$ vertices, and the singleton ones are selected completely.
	So, if there are $s$ singletons selected, then the number of paths is at least $s+\frac{n-s}{k-1}$ and the number of selected vertices is not less than $s+2\lceil \frac{n-s}{k-1} \rceil$.
	The worst case occurs when $s=0$ or $s=1$, depending on the residue of $n$ modulo $k-1$.
	This implies the claimed lower bound.
\end{proof}
	
	\begin{theorem}
	For all $k\ge 5$,
	
	$$
	  f(P_k) = \left\{
	    \begin{tabular}{lll}
	      $\dfrac{(k-1)^2}{2}$ && if\/ $k$ is \vspace{2ex} odd, \\
	      $\dfrac{(k-1)(k-2)}{2}+1$ & \qquad & if\/ $k$ is even.
	    \end{tabular}
	  \right.
	$$

	\end{theorem}

\begin{proof}
	The independence number of $P_k$ equals $\lceil k/2 \rceil$.
	So, as done in Theorem~\ref{treethree}, one can take a graph whose $\lceil k/2 \rceil-1$ (i.e., depending on the parity of $k$, either $(k-1)/2$ or $(k-2)/2$) components are $K_{k-1}$.
	This matches the claimed formula if $k$ is odd.
	For $k$ even, we can add a further isolated vertex.
	In all these graphs, every $k$-vertex induced subgraph contains a triangle (thus cannot be a subgraph of $P_k$), and is disconnected (hence $P_k$-free).
	This implies that $f(P_k)$ cannot be smaller than claimed.
	
	For an upper bound of the same value, let $G$ be a $P_k$-exact graph with $f(P_k)$ vertices.
	Should $G$ be $P_k$-free, we obtain by Lemma~\ref{l:linforest} that any graph with more vertices than the claimed formula contains an induced linear forest of order $k$, hence a proper subgraph of $P_k$.
	Indeed, for odd $k$,
	
	$$
	  2\left\lceil \frac{\frac{(k-1)^2}{2}+1}{k-1} \right\rceil = 2 \left\lceil \frac{k-1}{2} + \frac{1}{k-1} \right\rceil = k+1
	$$
	and for even $k$,
	
	$$
	  2\left\lceil \frac{\frac{(k-1)(k-2)}{2}+2}{k-1} \right\rceil = 2 \left\lceil \frac{k-2}{2} + \frac{2}{k-1} \right\rceil = k \,.
	$$	
	This contradicts our assumptions on $G$.
	So, $G$ contains one or more copies of $P_k$, all of which must be induced subgraphs.
	
	If $v_1\dots v_{k+1}$ is a path of order $k+1$ in $G$, then omitting any of its ends we obtain an induced $k$-path, but the omission of a vertex from the middle cannot yield a (disconnected) linear forest, hence the $(k+1)$-path actually induces a cycle.
	If there exists a further incident edge going out from the cycle, say $v_1w$, then for a similar reason, we obtain that $v_kw$ is also an edge.
	But then $v_{k+1}v_kwv_1v_2\dots v_{k-1}$ is a path with a chord $v_{k+1}v_1$, a contradiction.
	So $\{v_1,\dots,v_{k+1}\}$ induces a connected component.
	Then any further vertex with $v_1\dots v_{k-1}$ would induce disconnected a linear forest, therefore no such vertex can exist and $|G|=k+1$ would follow, far from being extremal if $k>4$.
	
	The final case, where the longest paths in $G$ have exactly $k$ vertices, is a little more complicated.
	Observe that
	
	$$
	  \frac{(k-1)(k-2)}{2}+1 < \frac{(k-1)^2}{2}
	$$
	holds for all $k>3$, so that the proof will be done if we show that $|G|$ is smaller than $\frac{1}{2}(k-1)(k-2)+1$ if $k\geq 5$.
	
	Let $P=v_1\dots v_k$ be a longest path; it is induced, and $P-v_i$ is a disconnected linear forest of order $k-1$ for every $2\leq i\leq k-1$.
	Thus, as $v_1$ and $v_k$ are of degree 1 by the non-extendability of $P$, every vertex $w$ not in $P$ has at least two neighors among the internal vertices of $P$.
	Choose a longest path $P'$ in $G-V(P)$, and let $w$ be an endpoint of $P$.
	Consider two neighbors of $w$ in $P$, say $v_i$ and $v_j$, $i<j$.
	We have $j\geq i+2$, otherwise $v_1\dots v_iwv_{i+1}\dots v_k$ would be a path longer than $P$.
	
	Observe that both $P'v_iv_{i+1}\dots v_k$ and $P'v_jv_{j-1}\dots v_1$ are non-induced paths, having at least one chord $v_jw$ and $v_iw$, respectively
	By assumption their order cannot exceed $k$, and since they are non-induced, they actually have order at most $k-1$.
	Consequently,
	
	\begin{eqnarray}
	 k-1 & \geq & |P'| + |P| - (i-1) \,, \nonumber \\
	 k-1 & \geq & |P'| + j \,. \nonumber
	\end{eqnarray}
	Substituting $|P|=k$ and summing up, we obtain
	
	$$
	  2k-2 \geq 2|P'|+k+(j-i)+1 \geq 2|P'|+k+3 \,.
	$$
	Hence,
	
	$$
	  |P'| \leq \left\lfloor \frac{k-5}{2} \right\rfloor .
	$$
	This means that the graph $G'=G-V(P)$ of order $|G|-k$ is $P_{\left\lfloor \frac{k-3}{2} \right\rfloor}$-free.
	(For $k=5$ we otained that $G\cong P_5$ must hold if $P_5\subseteq G$ and $G$ is a $P_5$-exact $P_6$-free graph!)
	On applying Lemma~\ref{l:linforest} it follows that the number of vertices in the largest induced linear forest $F$ in $G'$ is not smaller than
	
	$$
	  2 \left\lceil \frac{|G|-k}{\left\lfloor \frac{k-5}{2} \right\rfloor} \right\rceil - \epsilon_{|G|-k,\left\lfloor \frac{k-3}{2} \right\rfloor}
	  \geq
	  4 \cdot \frac{|G|-k}{k-5} - 1 \,.
	$$
	This $F$ can be extended to an induced disconnected linear forest in $G$ with at least two further vertices, namely $v_1$ and $v_k$.
	Therefore, $|F|\leq k-3$ because $G$ is $P_k$-exact.
	Consequently,
	
	$$
	  4 \cdot \frac{|G|-k}{k-5} - 1 \leq |F| \leq k-3 \,,
	$$
	$$
	  |G| \leq \frac{(k-2)(k-5)}{4} + k = \frac{k^2 - 3k + 10}{4}
	  <
	  \frac{(k-1)(k-2)}{2} + 1
	$$
	for all $k\geq 5$.
	This implies that a graph with maximum path length $k$ cannot be extremal for $f(P_k)$, hence completing the proof.	
\end{proof}

\subsection{Upper bound for $f(tH)$ and exact value for  $f(nK_2)$}

We first present an upper bound for $f(tH)$, for the graph consisting of $t$ vertex-disjoint copies  of a graph $H$.

\begin{proposition}
Let $H$ be a graph on $k \geq 3$  vertices and independence number $\alpha(H)$.  Then for $t\geq 2$, $f(tH) \leq R( H ,K_{tk}) + \left (\left \lceil \frac{tk}{\alpha(H)} \right \rceil - 1 \right )|H|  - 1$.   
\end{proposition}
\begin{proof}

Let $q = \left \lceil \frac{tk}{\alpha(H)} \right \rceil$ and consider a red-blue colouring of $E(K_n)$, for any  $n \geq R( H ,K_{tk}) + \left (\left \lceil \frac{tk}{\alpha(H)} \right \rceil - 1 \right )|H| $.   

If a blue $K_{tk}$ occurs,  we are done,  as the red graph $G$ contains an induced $E_{tk} \in \text{Com}^*(tH)$. Otherwise, by greedy deletion of copies of $H$, we have $q$ vertex-disjoint copies of $H$, each of which must be induced for otherwise we get a member of Com$^*(tH)$.  In this way  the graph   $F = qH$ is found, the union of all the $q$ induced copies of $H$.   If two vertex-disjoint copies of $H$ have an edge between them,  we have that $ F \in \text{Com}^*(tH)$.  

Now each copy of $H$ contains $\alpha(H)$ independent vertices, hence $F$ contains $q\alpha(H) =\alpha(H)\left \lceil \frac{tk}{\alpha(H)} \right \rceil \geq tk$ independent vertices and thus $G$ contains an induced $E_{tk}$.
\end{proof}

\noindent \textbf{Remark}: For $f(tK_k)$ we get $f(tK_k) \leq R(K_k , K_{tk}) +  (tk  - 1)k  - 1$.
Since in the particular case of $k= 2$ we trivially have $R(K_2,K_{2t})= 2t$, the upper bound reduces to $f(tK_2) \leq 2t +  (2t-1)2  - 1 = 6t - 3$\,;  and for $k=3$ we obtain $f(tK_3)  \leq R(K_3 ,K_{3t}) + (3t-1)3  - 1 = R(K_3,K_{3t}) + 9t - 4$.\\

Our last theorem concerns the exact determination of  $f(nK_2)$.  Recall the following result from \cite{LORIMER199291}:

\begin{theorem}
\label{R-match}
$R(nK_2, K_m) = 2n + m - 2$. Hence $R(nK_2, K_{2n}) = 4n-2$.
\end{theorem}

\begin{lemma}
\label{l:nK2}
Suppose $G$ contains $(n+1)K_2$, $n \ge 2$. Then $G$ contains a $2n$-vertex induced subgraph $F \in \text{Com}^*(nK_2)$.
\end{lemma}

\begin{proof}
Consider $M = (n+1)K_2$ with edges $e_j = (x_j, y_j)$, $j = 1, \dots, n+1$.
If $M$ is induced then $M^* = \{ e_1, \dots, e_{n-1} \} \cup (x_n, x_{n+1}) = (n-1)K_2 \cup 2K_1 \in \text{Com}^*(H)$.
Otherwise, without loss of generality, there is an edge $e = (x_1, x_2)$ and therefore taking the graph $F$ induced by $\{ e_1, \dots, e_{n} \}$, we have $F \in \text{Com}^*(nK_2)$.
\end{proof}

\begin{theorem}
$\mbox{ }$\
\begin{enumerate}
\item{$f(nK_2)=3n$ for $n=2,3$.}
\item{$f(nK_2)=4n-4$ for $n \geq 4$.}
\end{enumerate}
\end{theorem}
\begin{proof}

Note that for $n=4$, $3n=4n-4$,
\begin{enumerate}
\item{We begin with the lower bound.  For $n=2$, consider $G=2K_3$ on six vertices.  The only induced subgraphs on four vertices are $2K_2$ and $K_3 \cup K_1$, which are both in  $Leg(2K_2)$.  For $n=3$, we consider $3K_3$ on nine vertices.  The possible induced subgraphs on six vertices are $3K_2$, $K_3 \cup K_2 \cup K_1$ and $2K_3$ which are all in $Leg(3K_2)$.  Hence for $n=2,3$, $f(nK_2)\geq 3n$.

For the upper bound, we consider a graph $G$ on $3n+1$ vertices for $n=2,3$. Due to Theorem \ref{R-match} we have $R(nK_2,K_{2n}) =4n-2\leq 3n+1$ for $n \leq 3$, hence $G$ contains either induced $E_{2n} \in \text{Com}^*(nK_2)$ or $nK_2$.   So we can assume it contains $M=nK_2$ and this must be a largest matching in $G$ since $(n+1)K_2$ contains a $2n$-vertex induced subgraph $F \in \text{Com}^*(nK_2)$ by Lemma \ref{l:nK2}. 
Note that $M$ is induced.

Consider $B=V(G) \backslash V(M)$ which has $3n+1-2n=n+1$ vertices. By the maximality of $M$, $B$ must induce $E_{n+1}$. We now consider $n=2$ and $n=3$ separately. 
\smallskip

\noindent For $n=2$:
\begin{itemize}
\item{if a vertex $v \in M$ has no neighbour in $B$, then $B \cup \{v\}$ induces $E_4 \in \text{Com}^*(2K_2)$\,;}
\item{if a vertex $v \in M$ is adjacent to just one vertex in $B$, then $B \cup \{v\}$ induces $K_2 \cup 2K_1 \in \text{Com}^*(2K_2)$\,.}
\end{itemize}
Hence every vertex in $M$ must be adjacent to at least two vertices in $B$.  This implies that there are four vertices which form $P_4$ (not necessarily induced) which is in $\text{Com}^*(2K_2)$.  Hence $f(2K_2)<7$ and therefore $f(2K_2)=6$.

\noindent For $n=3$, let $M= \{ e_i=x_iy_i : i=1 ,\ldots ,3 \}$ and $B=\{u_1, \ldots,u_4\}$. Now 
\begin{itemize}
\item{if some vertex $u_j \in B$ is not adjacent to a vertex in $M$, say $y_3$ without loss of generality, then the set of vertices  $\{x_1,x_2,x_3,y_1,y_2,u_j\}$  induces $2K_2 \cup 2K_1 \in \text{Com}^*(3K_2)$\,;}
\item{if some vertex $u_j \in B$  is adjacent to two vertices from different edges in $M$, say $x_1$ and $x_2$ without loss of generality, then the set of vertices  $\{x_1,x_2,x_3,y_1,y_3,u_j\}$ induces	  $F \in \text{Com}^*(3K_2)$\,.}
\end{itemize}
Hence every vertex in $B$ must be adjacent to one or two vertices in $M$, and in the case of two vertices they must be $x_i$ and $y_i$. 

Since $|B|=4>3=|M|$, some two vertices say $u_1,u_2\in B$ have neigbours in the same matching edge say $x_1y_1$, and nonadjacent to the other two edges of $M$.
Then $\{u_1,u_2,x_2,x_3,y_2,y_3\}$ induces	  $2K_1\cup 2K_2 \in \text{Com}^*(3K_2)$\,.



Therefore $f(3K_2)<10$, implying $f(3K_2)=9$.}
\bigskip
\item{  Again, we begin with the lower bound.  Consider $G = K_{2n-1} \cup E_{2n-3}$, $|V(G)| = 4n - 4$. Any set on $2n$ vertices must include at least one isolated vertex and at least three vertices from the clique, inducing $K_p \cup E_{2n-p}$, $p \geq 3$ which is in $Leg(nK_2)$.

Hence $f(nK_2) \ge 4n - 4$.

For the upper bound, let $G$ be a graph on $4n-3$ vertices. We consider the following cases, showing in each case that  $G$ contains an induced subgraph $F$ on $2n$ vertices with $F \in \text{Com}^*(nK_2)$.

\begin{enumerate} 
\item{If $G$ contains $(n+1)K_2$, then by Lemma 1, $G$ contains a $2n$-vertex induced subgraph $F \in \text{Com}^*(nK_2)$. }
\item{If $G$ contains $M = nK_2$ as its largest matching with edges $e_i = (x_i, y_i)$, $i = 1, \dots, n$, it must be induced.  Let $B = V(G) \setminus V(M)$, $|B| = 4n-3 - 2n = 2n-3$.  Since $M$ is a maximal matching, $B$ must induce $E_{2n-3}=\{u_i : i=1, \ldots, 2n-3\}$.

\begin{itemize}
\item{If there is a vertex $u_j\in B$ having no neighbour in $M$, then the vertices $ x_i, y_i$ $(i = 1, \dots, n-1)$, $x_n$, $u_j$  induce $(n-1)K_2 \cup 2K_1 \in \text{Com}^*(nK_2)$.}

\item{If any vertex say $u_1 \in B$  is adjacent to two vertices from different edges in $M$, say $x_1$ and $x_2$ without loss of generality, then the set of vertices  $x_1$, $u_1$, $x_i$, $y_i$ ($i = 2, \dots, n$), $x_n$, $u_1$ induces a subgraph properly containing $nK_2$, hence a member of $\text{Com}^*(nK_2)$\,. }
\end{itemize}

Hence every vertex in $B$ must be adjacent to one or two vertices in $M$, and in the case of two vertices they must be $x_i$ and $y_i$. 
Similarly as in the case of $3K_2$, here $|B|=2n-3>n=|M|$, therefore two vertices from $B$ have their neighbourhood in the same matching edge, with no neighbours in any of the other $n-1$ edges of $M$.
Thus, also here an induced subgraph $(n-1)K_2 \cup 2K_1 \in \text{Com}^*(nK_2)$ occurs.}

%

\item{We are left with the case $|M|\leq n-1$.
We now observe that $4n - 3  = R( (n-1)K_2 , K_{2n} ) +1$.
Since $|V(G)|= 4n-3$, $G$ contains either $( n-1)K_2$  or $E_{2n} \in  \text{Com}^*(nK_2)$.  So, we may assume the largest matching in $G$ is $M  = \{ e_1,\ldots,e_{n-1}\}$.  Let $B = V(G) \setminus V(M)$, $|B| = 4n-3 - 2(n-1)  = 2n-1$.  Again, by maximality of $M$, $B$ induces $E_{2n-1}=\{u_i: i=1,\ldots 2n-1\}$.

If there is a vertex say $x_1 \in M$ not adjacent to any vertex in $B$, then $B \cup \{x_1\}$ induces $ E_{2n} \in \text{Com}^*(nK_2)$.  
Similarly, if there is a vertex say $x_1 \in M$  adjacent to just one vertex in  $B$, then $B \cup  \{x_1\}$  induces $K_2  \cup  (n-2)K_1 \in \text{Com}^*(nK_2)$.

Hence it remains to assume that every vertex in $M$ is adjacent  to at least two vertices in  $B$.    
Consider now two edges from $x_1$ to $B$ and two edges from $y_1$ to $B$.
mong those four edges, there are two forming $2K_2$.
Replacing the edge $x_1y_1$ with those two, together with the other $n-2$ edges of $M$ we obtain a matching of size $n$, contradicting the assumption that $M$ is a lergest matching in $G$.}

\end{enumerate}

Hence $f(nK_2)=4n-4$ for $n \geq 4$.}

\end{enumerate}
\end{proof}

\section{Conclusion}

In this paper we were mainly interested in the following problem: \\

\noindent Given a graph $H$, determine $f(H) = \max\{n : \exists G \in \text{Exact}(H), |V (G)| = n\}$.\\

We have presented many results on $f(H)$.  We consider the following open problems as a starting point for future research.  .
\begin{enumerate}
\item{For every tree $T$, determine $f(T)$ or at least prove that  \[(\alpha(T) - 1)(|V(T)| - 1) \leq f(T) \leq (\alpha(T) - 1)|V(T)|.\]}
\item{Prove that  $f(G) \leq R(G,K_k)$ holds for every graph $G \notin \{K_k. E_k\}$ on $k \geq 3$ vertices.}
\end{enumerate}

\paragraph{Acknowledgement.}

Research of the second author was supported in part by the
 ERC Advanced Grant ``ERMiD''.

\bibliographystyle{plain}
\bibliography{Exactbib}
 
\end{document}